\documentclass[a4paper]{amsart}%,fleqn]{article}

\usepackage[utf8]{inputenc}

\newtheorem{theorem}{Theorem}
\newtheorem{lemma}[theorem]{Lemma}
\newtheorem{proposition}[theorem]{Proposition}
\newtheorem{remark}[theorem]{Remark}

\title[A nonautonomous superlinear problem]{Existence and classifiction of radial solutions of a  nonlinear nonautonomous Dirichlet problem}
\date{}
\author[M. Rouaki]{Mohamed Rouaki}
\address{Department of Mathematics, University of Blida, Algeria.}
\email{rouakimd@yahoo.fr}

\begin{document}
\begin{abstract}
This paper generalizes a classification of solutions of a superlinear Dirichlet problem given in \cite{rouaki2} to a nonautonomous case.
In \cite{rouaki1} the increasing of $f(t)$ was used to prove the classification and in \cite{rouaki2} the unicity of the solution of the \emph{Cauchy} problem was used. Here the classification appears as a consequence of the \emph{a priori} estimates. It results that existence classificarion remain true for a class of nonautonomous problems.
\end{abstract}

\maketitle
\section{Introduction}
We are interested by radial solutions of the nonautonomous problem
\begin{equation}
	-\Delta u = g(u) - \lambda -f(x), \textrm{ on } \Omega \textrm{ and } u=0 \textrm{ on } \partial \Omega
\end{equation}
where $\Omega$ denotes the unit ball in $\mathbf{R}^n$, $\lambda >0$, $f$ is a $C^1$ radial function on $\Omega$. $g\in C^{0,\alpha} (\mathbf{R},\mathbf{R})$ and there exists $A>0$ such
that $g_+=g \big| _{[A,\infty [ }$\ is positive, increasing,
differentiable and convex, $g_{-}=g \big| _{]-\infty ,-A]}$\ is positive, convex
and decreasing. In addition 
\begin{equation}
	\lim \frac{g(x)}{x} =\pm \infty ,\quad x \to \pm \infty
\end{equation}
\begin{equation}
	\lim \sqrt{\frac{R(x)}{x}}\ \frac{g^{-1} _+(x)}{g^{-1} _-(x)} =\pm \infty ,\quad x \to \pm \infty
\end{equation}

A classical problem of the existence of radial solutions still interesting in \emph{superlinear} case see \cite{dambrosio} and \cite{grossi}.

For the positone problem different methods have been used \cite{grossi}, and for the nonpositone problem, radial solutions have been considered using the shooting method \cite{dambrosio}\cite{castro}. Here we deal with the \emph{nonpositone} problem using the homotopy of the topological degree \cite{leray}.

P. L. Lions in \cite{lions} notes that many existence results of nodal solutions have been obtened but no classification of solutions have been given.

Remark that a classification of solutions set was introduced by \\ P.H. Rabinowitz \cite{rabinowitz} based on the number of zeros of the solution $u(t)$ to prove existence results for a semilinear Sturm-liouville problem. 

In this paper we use the homotopy of the topological degree and a classification of solutions based on the number of zeros of the second hand side of Eq.(1) \ $g(u(t))-\lambda -f(t) = 0$, $t\in \mathbf{R}$. This approch represents an alternative for the shooting method and have been used in  \cite{rouaki1}\cite{rouaki2}.

This paper generalize the existence result given in \cite{rouaki2} for a nonautonomous case. The main result is the Proposition(3) in which the classification of solutions set appears as a consequence of the \emph{a priori} estimates. Indeed in \cite{rouaki1} the classification was given by the increasing property of $f(t)$, see proof of Proposition(3) Eq.(2.19), and in \cite{rouaki2} the unicity of the solution of the \emph{Cauchy} problem was used, see Proposition(4) \cite{rouaki2}.

Remark that the topological method is not limited by \emph{critic Pohozaev-Sobolev exponent} but only by \emph{a priory} estimates. Hence, the existence result given in Theorem (1)  \cite{rouaki2} depends only on conditions (2) and (3) and stills valid for $\mathbf{R^n}$, $n\ge 1$. To our knowledge the most general existence results known at this time for nodal solutions of \emph{nonpositone} Elliptic problems are subject to the limite of \emph{critic Pohozaev-Sobolev exponent}.

A remarkable \emph{a priori} estimates for positive solutions of elliptic problems was given in \cite{figueiredo} and used to get existence result with the topological degree.

Here, properties of the \emph{nonpositone} problem and nodal solutions have been exploted to get  an \emph{a priori} estimates which is not limited by the \emph{critic Pohozaev-Sobolev exponent}.

The plane of the proof is similar to \cite{rouaki2} and most arguments of proofs remain true for (1). So we will give details just for the proof of Proposition(\ref{p_classification}) which generalizes Proposition(4) in \cite{rouaki2}.

\section{Existence and classification of solutions}
We consider the problem
\begin{equation}
	-u''(t)-\frac{n-1}{t}u'(t) =g(u(t))-\lambda -f(t)
\end{equation} 
	\[ u'(0) = 0, u(1)=0 
	\]
$u$ having a local minimum in \emph{zero}. This is a non autonomous problem related to (5) in \cite{rouaki2}. In addition suppose that $f\in C^{1}([0,1],\mathbf{R})$.

Recall that $\lambda >0, g\in C(\mathbf{R,R})$\ and there exists $A>0$ such
that $g_+=g \big| _{[A,\infty [ }$\ is positive, increasing,
differentiable and convex, $g_{-}=g \big| _{]-\infty ,-A]}$ \ is positive, convex
and decreasing. In addition 
\begin{equation}
	\lim \frac{g(x)}{x} =\pm \infty ,\quad x \to \pm \infty
\end{equation}

Let $k\in N,$\ $\lambda >A,$ $E=\{u\in C^{1}([0,1],\mathbf{R}): u'(0) \le 0, u(1)=0\}$ and $Z_{k}(\lambda )$ a subset of $E$ defined by 
	\[ Z_{k}(\lambda )=\left\{ u\in E:u(t)-g_+^{-1}(\lambda +f(t))\textrm{ has }k \textrm{ simple zeros in\ }[0,1]\right\}
	\]

We denote $M= \| f \| _{C^1}$.

\vspace{10pt}
The following proposition recalls the \emph{a priori} estimate given in proposition (2) in \cite{rouaki2}.

\paragraph{\textbf{Proposition.}}
There exist $C>0$ and $K(\lambda )$ a continuous function defined on 
$[C,\infty [ $ such that, for each solution $(u, \lambda )$ of
(1) satisfying $\lambda >C$ and $u'(0)\leq 0,$ we have 
$\| u \| <K(\lambda )$.

For a local maximum $\beta$
	\[	u(\beta )<2R(4(\lambda +M)), \quad R(x) = \max \{ |g^{-1}_- (x)|, |g^{-1}_+ (x)| \}
	\]
and for a local minimum $\alpha$
	\[	|u(\alpha)| \le R(\lambda + M )
	\]

The proof of the propostion is the same as proof of Proposition(2) in \cite{rouaki2}.

\paragraph{\textbf{Some general formulas}} \

--- The mean theorem gives 
	\[	\left| \int_{g_+^{-1}(\lambda +m_{1})}^{g_+^{-1}(\lambda
		+m_{2})}(g(u)-\lambda )du\right| \leq m_{2}\left| g_+^{-1}(\lambda
		+m_{2})-g_+^{-1}(\lambda +m_{1})\right|
	\]
and gives $\mu \in ]m_{1},m_{2}[$
	\[	g_+^{-1}(\lambda +m_{2})-g_+^{-1}(\lambda +m_{1}) = \frac{m_{2}-
		m_{1}}{g'(g_+^{-1}(\lambda +\mu ))}
	\]
\begin{equation}
\left| \int_{g_+^{-1}(\lambda +m_1)}^{g_+^{-1}(\lambda
+m_2)}(g(u)-\lambda )du\right| \le m_2 \left| \frac{m_2 - m_1}
{g'(g_+^{-1}(\lambda +\mu ))}\right|
\end{equation}

--- for $x > a$ large enough $g_+$ is convex then
	\[ g'(x) > \frac{g(x)-g(a)}{x-a} 
	\]
for $x$ large enough there exists $\gamma >0$ such that 
	\[	g'(x) > \gamma g(x)/x 
	\]
set $x= g^{-1}(\lambda +\mu)$ to get
\begin{equation}
\frac{1}{g^{\prime }(g_+^{-1}(\lambda +\mu ))}\to 0,\lambda
\to +\infty
\end{equation}

--- Let $a,b\in [ 0,1]$
\begin{eqnarray}
\int _a ^b f u' dt &=&(f(b)u(b)- f(a)u(a))+\int_{a}^{b} f' u dt  \nonumber \\
\left| \int_a ^b f u' dt \right| &\leq & 3M \max | u(t)|
	\quad \le 6M\,R(4(\lambda +M))
\end{eqnarray}

--- The concavity of $g_+^{-1}$ implies that for $x>\alpha$ large enough 
	\[	\frac{g_+^{-1}(x) - g_+^{-1}(\alpha )}{x-\alpha }
	\] 
is decreasing, then for $b>a>0$ and $\lambda $ large enough 
	\[	\frac{g_+^{-1}(b\lambda ) - g_+^{-1}(\alpha )}{b\lambda -\alpha } <
		\frac{g_+^{-1}(a\lambda ) - g_+^{-1}(\alpha )}{a\lambda -\alpha }
	\]
we deduce that there exists $\gamma > 0$ such that 
	
\begin{equation}
	g_+^{-1}(b\lambda ) < \gamma g_+^{-1}(a\lambda )
\end{equation}

\begin{remark}
Increasing of $g$ gives, for $\lambda $ large enough, $u>g_+^{-1}(\lambda + f)$ implies 
$g(u)-(\lambda + f)>0,$ and $u=g_+^{-1}(\lambda + f)$ implies 
$g(u)-(\lambda + f)=0,$ hence $0\leq u<g_+^{-1}(\lambda + f)$
implies $g(u)-(\lambda + f)<0$.

For $\beta $ a local maximum \quad $g(u(\beta)) - (\lambda +f(\beta)) \ge 0$, from which  \\
$g_+^{-1}(\lambda + f(\beta ))\leq u(\beta )$. (contrapositive of the last implication)

For $\alpha $ a positive local minimum $g_+^{-1}(\lambda + f(\alpha ))\geq u(\alpha )$.
\end{remark}

The following lemma genaralizes Lemma(5) in \cite{rouaki2} which is used in the following to prove Proposition(3).

Estimation of the derivative at zeros of \ $u(t)-g_+^{-1}(\lambda + f(t))$.
\begin{lemma}
There exists a sequence $(A_k)$ $(k\geq 1)$ of positive numbers such that,
for each solution $(u, \lambda )$ of (1) satisfying $u'(0)\leq 0$, 
$\lambda >A_{k}$ and \ $u-g_+^{-1}(\lambda + f)$ having at least \ $k$ zeros, 
there exist $B>0$ satisfying for the $k$ largest zeros 
	\[	|u'(\tau )|>B\sqrt{\lambda g_+^{-1}(\lambda /2)}
	\]
\end{lemma}

\begin{proof}
\emph{
The $k$ largest zeros of $u-g_+^{-1}(\lambda + f)$ 
are denoted by $\tau _{1}<\tau _{2}<..<\tau _{k}<1$. $\tau_k$ represents the largest zero.}

\emph{Estimation of $u'(\tau _{k})$.}

Let $\eta \in ]\tau _{k},1[$ be the smallest zero of $u(t).$ Since $%
g_+^{-1}(\lambda + f)>u,$ from remark(1) $u$ has no local maximum in 
$]\tau _{k},\eta [ $, then it is decreasing on $[\tau _{k},\eta ]$ and from (1) it is convex.

Let $a$ be the unique element of $]\tau _{k},\eta [ $ such that 
$u(a)=g_+^{-1}(\lambda /2)$. Denoting by $h(t)$ the segment joining $u(a)$
and $u(\eta )=0$, and setting \ $v(t)=h(t)-u(t)$ on \ $]a,\eta [ ,$ then \
$-v''=\lambda + f-g(u)- pu'.$ Since $u<g_+^{-1}(\lambda /2)$ 
and is decreasing \quad $-v''>\lambda /2$,
since $v<g_+^{-1}(\lambda /2)$
	\[	-v''>\frac{\lambda }{2g_+ ^{-1}(\lambda /2)}v,
		\textrm{ on }]a,\eta [	
	\]
	\[	v(\eta )=v(a)=0%
	\]
setting $t=s\left( \eta -a\right) +a, s\in [ 0,1]$ and 
$ w(s)=v(s \left( \eta -a \right) +a)$
	\[	-w''>\left( \eta -a\right) ^{2}\frac{\lambda }{2g_+^{-1}(\lambda /2)} w, \textrm{ on }]0,1[
	\]
	\[	w(0)=w(1)=0
	\]
the comparison theorem of Sturm gives \, $(\eta -a) <\sqrt{2} \pi \sqrt{g_+^{-1}(\lambda /2)\big/ \lambda }$.

Since $u$ is convex on $]\eta ,\tau _{k}[$, \, $\left| u'(\tau
_{k})\right| >\left| u'(a)\right| > u(a)\big/(\eta -a)$, \newline
hence $\left| u'(\tau _{k})\right| >\frac{1}{\sqrt{2}\pi } 
\sqrt{\lambda g_+^{-1}(\lambda /2)}$.

\emph{We shall use the recurrence argument.}

Let $B>0$, $\delta >0$, $\tau _{i}$
and $\tau _{i+1}$ two consecutive zeros such that $\left| u'(\tau
_{i+1})\right| >B\sqrt{\lambda g_+^{-1}(\lambda /2)},$ then for $\lambda $
large enough we have $\left| u'(\tau _{i})\right| >(B-\delta )\sqrt{%
\lambda g_+^{-1}(\lambda /2)}$. 

Indeed, multiplying (1) by $u'$ and integrating to get 
	\[	\frac{u'^{2}(\tau _{i})}{2} \ge \frac{u'^{2}(\tau _{i+1})}{2} + \!
		\int _{u(\tau _{i})}^{u(\tau _{i+1})} \!\!\!\! (g(u)-\lambda )du -
		\! \int_{\tau _{i}}^{\tau _{i+1}} \!\!\!\! f u'dt
	\]
then (3,5) give
	\[	\frac{u'^{2}(\tau _{i})}{2}\geq \frac{B^{2}}{2}\lambda
g_+^{-1}(\lambda /2)-M \left| \frac{ f(\tau _{i+1})- f(\tau
_{i})}{g^{\prime }(g_+^{-1}(\lambda ))}\right| - 6M\,R(4(\lambda +M))
	\]
(2,6) give $\frac{R(4(\lambda +M))}{\lambda g_+^{-1}(\lambda /2)}%
\to 0$ and (4) gives $\frac{ f(\tau _{i+1})- f(\tau
_{i})}{g'(g_+^{-1}(\lambda ))} \to 0$.
\end{proof}

\vspace{10pt}
The following proposition genaralizes the Proposition(4) in \cite{rouaki2}.
\begin{proposition}\label{p_classification}
There exists a sequence $(B_{k})(k\geq 0)$ of positive numbers such that,
for each $\lambda >B_{k}$, (1) has no solution $u \in \partial Z_{2k}(\lambda )$ satisfying $u'(0)\leq 0$.
\end{proposition}

\begin{proof}
\emph{By contradiction, let }$u\in \partial Z_{2k}(\lambda )$ 
\emph{be a solution of (1)}.

\emph{Case }$k=0$: Let $(v_{n})$ a sequence of solutions in 
$Z_{0}(\lambda )$ such that $v_{n}\to u$. From remark(1) 
$v_{n}<0$ on $]0,1[$ then $u\leq 0$, hence $u-g_+^{-1}(\lambda + f)$
has no zero from which $u\in Z_0(\lambda )$ thus $u\notin \partial Z_0(\lambda )$, contradiction.

\emph{Case }$k\geq 1$: \emph{First, we shall prove that }
$u-g_+^{-1}(\lambda + f)$ \emph{ has at most } $2k$ \emph{ simple
zeros}.

Indeed, let $\tau $ be a simple zero, then there exist $\epsilon _{0}>0$, 
$\epsilon _{1}>0$ and $\delta >0$ such that $\tau $ is the unique zero on 
$]\tau -\epsilon _{0},\tau +\epsilon _{0}[$, (one assume that 
$u-g_+^{-1}(\lambda + f)$ is increasing. If it is decreasing the
inequalities are inverse and the proof is similar)
	\[	\left| 
\begin{array}{l}
u(\tau -\epsilon _{0})-g_+^{-1}(\lambda + f(\tau -\epsilon
_{0}))<-\epsilon _{1} \\ 
u(\tau +\epsilon _{0})-g_+^{-1}(\lambda + f(\tau +\epsilon
_{0}))>\epsilon _{1} \\ 
u'-\left[ g_+^{-1}(\lambda + f)\right] ^{\prime }>\delta%
\end{array}
\right. 
	\]
Let $(v_{n})$ be a sequence of $Z_{2k}(\lambda )$ such that $%
v_{n}\to u$ in $E,$ there exists $n(\epsilon _{0},\epsilon
_{1},\delta )\in N$ such that for $n>n(\epsilon _{0},\epsilon
_{1},\delta )$ 
	\[	\left| 
\begin{array}{l}
\left| u(\tau -\epsilon _{0})-v_{n}(\tau -\epsilon _{0})\right| <\epsilon _1 /2 \\ 
\left| u(\tau +\epsilon _{0})-v_{n}(\tau +\epsilon _{0})\right| <\epsilon _1 /2 \\ 
\left\| u'-v'_{n}\right\| _{\infty }< \delta / 2
\end{array}
\right. 
	\]
from which 
	\[	\left| 
\begin{array}{l}
v_{n}(\tau -\epsilon _{0})-g_+^{-1}(\lambda + f(\tau -\epsilon
_{0}))< -\epsilon _1 /2 \\ 
v_{n}(\tau +\epsilon _{0})-g_+^{-1}(\lambda + f(\tau +\epsilon
_{0}))> \epsilon _1 /2 \\ 
v'_{n}-\left[ g_+^{-1}(\lambda + f)\right] ^{\prime }> \delta /2
\end{array}
\right.
	\]
which implies that $v_{n}-g_+^{-1}(\lambda + f)$ has a unique simple
zero on $]\tau -\epsilon _{0},\tau +\epsilon _{0}[$. Since $v_{n}\in
Z_{2k}(\lambda )$, $v_{n}-g_+^{-1}(\lambda + f)$ has exactly $2k$
simple zeros, then $u-g_+^{-1}(\lambda + f)$ has at most $2k$ simple
zeros.

\emph{There are not exactly $m$ simple zeros with $m<2k$}.

Indeed, by contradiction assume that $u\in Z_{m}(\lambda )$. Since 
$Z_{2k}(\lambda )$ and $Z_{m}(\lambda )$ are open sets of $E$ and 
$Z_{m}(\lambda )\cap Z_{2k}(\lambda )\neq \emptyset $, then $Z_{m}(\lambda
)\cap \partial Z_{2k}(\lambda )=\emptyset $, contradiction.

Last, since there exist at most $2k$ simple zeros of $u-g_+^{-1}(\lambda
+ f)$ there exists $\tau _{j}$ a zero which is not simple $j\leq
2k+1$. From the lemma (2), there exists $A_{2k+1}>0$ such that for $\lambda
>A_{2k+1}$ \, $|u'(\tau _{j})| >B \sqrt{\lambda g_+^{-1}(\lambda /2)}$

On the other hand \, $[ g_+^{-1}(\lambda +
f)]' = \frac{f'}{g'_+(g_+^{-1}(\lambda +f))},$  (4) implies that 
$|u'(\tau _j)| > \left| \big(g_+^{-1}(\lambda + f(\tau  _i))\big) ' \right|$
for $\lambda $ large enough, then $\tau _j$ is a simple zero of \
$u-g_+^{-1}(\lambda + f)$, contradiction.
\end{proof}

\end{document}